\documentclass[reqno]{amsart}
\makeatletter
\let\origsection=\section 
\def\section{\@ifstar{\origsection*}{\mysection}}
\def\mysection{\@startsection{section}{1}\z@{.7\linespacing\@plus\linespacing}{.5\linespacing}{\normalfont\scshape\centering\S}}
\makeatother

\usepackage{amsmath,amssymb,amsthm}
\usepackage{mathrsfs}
\usepackage{dsfont}
\usepackage{mathabx}\changenotsign
\usepackage{enumerate}
\usepackage{microtype}
\usepackage{color}

\usepackage{xcolor}
\usepackage[backref=page]{hyperref}
\hypersetup{%
    colorlinks,
    linkcolor={red!60!black},
    citecolor={green!60!black},
    urlcolor={blue!60!black}
}

\usepackage{tikz}
\usepackage{subcaption}
\usetikzlibrary{graphs}

\usepackage{bookmark}

\usepackage[abbrev,msc-links,backrefs]{amsrefs}
\usepackage{doi}

\renewcommand{\PrintDOI}[1]{\doi{#1}}

\usepackage[T1]{fontenc}
\usepackage{lmodern}

\usepackage[english]{babel}

\usepackage{geometry}
\usepackage{comment}
\usepackage{enumitem}

\def\itm#1{\rm ({#1})}
\def\itmit#1{\itm{\it #1\,}}
\def\rom{\itmit{\roman{*}}}

\newcommand{\oldqed}{}
\def\endofFact{\hfill\scalebox{.6}{$\Box$}}
\newenvironment{claimproof}[1][Proof]{
  \renewcommand{\oldqed}{\qedsymbol}
  \renewcommand{\qedsymbol}{\endofFact}
  \begin{proof}[#1]
}{
  \end{proof}
  \renewcommand{\qedsymbol}{\oldqed}
} 

\let\polishlcross=\l
\def\l{\ifmmode\ell\else\polishlcross\fi}

\def\cP{\mathcal{P}}
\def\PP{\mathbb{P}}

\let\setminus=\smallsetminus

\makeatletter
\def\moverlay{\mathpalette\mov@rlay}
\def\mov@rlay#1#2{\leavevmode\vtop{%
   \baselineskip\z@skip \lineskiplimit-\maxdimen
   \ialign{\hfil$\m@th#1##$\hfil\cr#2\crcr}}}
\newcommand{\charfusion}[3][\mathord]{
    #1{\ifx#1\mathop\vphantom{#2}\fi
        \mathpalette\mov@rlay{#2\cr#3}
      }
    \ifx#1\mathop\expandafter\displaylimits\fi}
\makeatother

\newtheorem{theorem}{Theorem}
\newtheorem{lemma}[theorem]{Lemma}

\newtheorem{claim}[theorem]{Claim}
\newtheorem{subclaim}[theorem]{Sub-Claim}
\newtheorem{proposition}[theorem]{Proposition}
\newtheorem{definition}[theorem]{Definition}
\newtheorem{fact}[theorem]{Fact}

\usepackage{stmaryrd}
\usepackage{centernot}

\newcommand\rb{\text{{\tiny\rm rb}}}

\newcommand{\mcarrow}{\xrightarrow{\rb}}
\newcommand{\nmcarrow}{{\centernot{\xrightarrow{\rb \, \,}}}}

\def\pmc#1{p^{\rm rb}_{#1}}

\usepackage{datetime}
\usepackage{lineno}
\newcommand*\patchAmsMathEnvironmentForLineno[1]{%
\expandafter\let\csname old#1\expandafter\endcsname\csname #1\endcsname
\expandafter\let\csname oldend#1\expandafter\endcsname\csname end#1\endcsname
\renewenvironment{#1}%
{\linenomath\csname old#1\endcsname}%
{\csname oldend#1\endcsname\endlinenomath}}%
\newcommand*\patchBothAmsMathEnvironmentsForLineno[1]{%
\patchAmsMathEnvironmentForLineno{#1}%
\patchAmsMathEnvironmentForLineno{#1*}}%
\AtBeginDocument{%
\patchBothAmsMathEnvironmentsForLineno{equation}%
\patchBothAmsMathEnvironmentsForLineno{align}%
\patchBothAmsMathEnvironmentsForLineno{flalign}%
\patchBothAmsMathEnvironmentsForLineno{alignat}%
\patchBothAmsMathEnvironmentsForLineno{gather}%
\patchBothAmsMathEnvironmentsForLineno{multline}%
}

\begin{document}
\shortdate
\yyyymmdddate
\settimeformat{ampmtime}
\date{\today, \currenttime}

\title[The anti-Ramsey threshold of complete graphs]%
{The anti-Ramsey threshold of complete graphs}

\author[Y.~Kohayakawa]{Yoshiharu Kohayakawa}
\address{(Y. Kohayakawa and G. O. Mota) Instituto de Matem\'atica e Estat\'{\i}stica, Universidade de S\~ao Paulo, S\~ao Paulo, Brazil}
\email{yoshi@ime.usp.br, mota@ime.usp.br}
\author[G.~O.~Mota]{Guilherme Oliveira Mota}

\author[O.~Parczyk]{Olaf Parczyk}
\address{(O. Parczyk) Department of Mathematics and Computer Science, Freie Universität Berlin, Berlin, Germany }
\email{parczyk@mi.fu-berlin.de}

\author[J.~Schnitzer]{Jakob Schnitzer}
\address{(J.~Schnitzer) Fachbereich Mathematik, Universit\"at Hamburg, Hamburg, Germany}
\email{jakob.schnitzer@uni-hamburg.de}

\thanks{\rule[-.2\baselineskip]{0pt}{\baselineskip}%
  The first author was partially supported by CNPq (311412/2018-1,
  423833/2018-9, 406248/2021-4) and FAPESP (2018/04876-1,
  2019/13364-7).
  The second author was partially supported by FAPESP (2018/04876-1,
  2019/13364-7) and CNPq (304733/2017-2, 428385/2018-4). 
  The third author was supported by DFG grant PE
  2299/1-1.
  The cooperation of the authors was supported by a joint CAPES/DAAD
  PROBRAL project (Proj.~430/15, 57350402, 57391197).
  This research was partially supported by CAPES (Finance Code 001).
  FAPESP is the S\~ao Paulo Research Foundation.  CNPq is the National
  Council for Scientific and Technological Development of Brazil.%
}

\begin{abstract}
  For graphs~$G$ and~$H$, let~$G\mcarrow H$ denote the property that,
  for every proper edge-colouring of~$G$, there is a rainbow~$H$
  in~$G$.  For every graph~$H$, the threshold
  function~$\pmc H=\pmc H(n)$ of this property in the random
  graph~$G(n,p)$ satisfies~$\pmc H=O(n^{-1/m^{(2)}(H)})$, where
  $m^{(2)}(H)$ denotes the so-called maximum $2$-density of $H$.
  Completing a result of Nenadov, Person, {\v
    S}kori{\'c}, and Steger [{\frenchspacing J.~Combin.\ Theory Ser.~B
    \textbf{124} (2017), 1--38}], we prove a matching lower bound for
  $\pmc {K_k}$ for $k\geq 5$.  Furthermore, we show that
  $\pmc {K_4} = n^{-7/15} \ll n^{-1/m^{(2)}(K_4)}$.
\end{abstract}


\maketitle
\footskip=28pt
\pagestyle{plain}

\section{Introduction}\label{sec:Introduction}

As usual, we write $G(n,p)$~for the binomial random graph and, given a
monotone increasing property~$\cP$ of graphs, we call a function
$\hat p=\hat p(n)$ a \emph{threshold function} for~$\cP$ if
\begin{equation}
  \label{eq:1}
  \lim_{n\to\infty}\PP(G(n,p)\in\cP)=
  \begin{cases}
    0 &\text{if $p\ll\hat p$}\\
    1 &\text{if $p\gg\hat p$}.
  \end{cases}
\end{equation}
In~\eqref{eq:1} and in what follows, $f\ll g$ stands for
$\lim_nf(n)/g(n)=0$ and $f\gg g$ stands for $g\ll f$.  Following
standard practice, we refer to any~$\hat p$ as in~\eqref{eq:1} as
\textit{the} threshold function for~$\cP$ even though threshold
functions are not uniquely defined.  When $\lim_n\PP(G(n,p)\in\cP)=1$,
we say that~$G(n,p)$ has~$\cP$ \textit{with high probability}.  For
notation and terminology not explicitly defined here,
see~\cite{Bo98,JLR00}.

In a major breakthrough in the area of random graphs, R\"odl and
Ruci\'nski~\cite{RoRu93, RoRu95} established the threshold function
for the property $G(n,p)\to(H)_r$ for any given graph~$H$, that is,
the property that, for every edge-colouring of~$G(n,p)$ with at
most~$r$ colours, there exists a monochromatic copy of~$H$
in~$G(n,p)$.  Their result is in fact more refined, but it implies
that, as long as~$H$ is not a star-forest, the threshold
is~$n^{-1/m^{(2)}(H)}$, where $m^{(2)}(H)$ stands for the
\emph{maximum $2$-density} of~$H$, given by
\begin{equation}
  \label{eq:2}
  m^{(2)}(H) = \max\left\{{|E(J)|-1\over |V(J)|-2}\colon J\subseteq
    H,\,|V(J)|\ge 3 \right\}
\end{equation}
(for simplicity, we suppose~$|V(H)|\geq3$ in~\eqref{eq:2}).
Since this result was obtained, this line of research developed into a very
rich area.  We refer the reader to~\cite[Chapter~8]{JLR00} and
Conlon~\cite{Co14} for an overview of this line of research.

In this paper we study a variation of the classical Ramsey property
described above. Given graphs~$G$ and~$H$, we denote by $G\mcarrow H$
the so-called \emph{anti-Ramsey} property of~$G$ with respect to~$H$:
for every \textit{proper} edge-colouring of~$G$ there exists a
\textit{rainbow copy} of~$H$ in~$G$, i.e., a copy of~$H$ where no two
edges are assigned the same colour.  Note that the property
$G\mcarrow H$ is monotone in~$G$ for every fixed graph $H$, and hence
it admits a threshold~\cite{BoTh87}, which we denote by $\pmc H$.

The study of anti-Ramsey properties of random graphs was initiated by
R\"odl and Tuza~\cite{RoTu92}, who proved that for every integer
$\ell\geq 4$ we have $G(n,p)\mcarrow C_\ell$ with high probability if
$p \ge n^{-(\ell-2)/(\ell-1)} = n^{-1/m^{(2)}(C_\ell)}$.  The general case was studied
in~\cite{KoKoMo12}, where the following result was proved.

\begin{theorem}
  \label{thm:conditional}
  Let $H$ be a graph.  Then there exists a constant $C>0$ such that
  for $p=p(n)\geq Cn^{-1/m^{(2)}(H)}$ we have ${G(n,p)\mcarrow H}$
  with high probability.
\end{theorem}

Note that Theorem~\ref{thm:conditional} implies that
$\pmc H\leq n^{-1/m^{(2)}(H)}$.  However, in contrast to the
R\"odl--Ruci\'nski case, it was established in~\cite{KoKoMo16+} that
there are infinitely many ``non-trivial'' graphs~$H$ for which
$\pmc H\ll n^{-1/m^{(2)}(H)}$, highlighting the interest on the
problem of establishing \textit{lower bounds} for~$\pmc H$.

Nenadov, Person, {\v S}kori{\'c} and Steger~\cite{NePeSkSt14}
developed a general framework for proving lower bounds for various
Ramsey type problems within random settings. Applying this framework
to the anti-Ramsey problem, they proved that the upper bound given in
Theorem~\ref{thm:conditional} is sharp for sufficiently large cycles
and complete graphs.

\begin{theorem}[Nenadov, Person, {\v S}kori{\'c} and
  Steger~\cite{NePeSkSt14}]
  \label{thm:cycles-complete-graphs}
  Let $H$ be a cycle on at least $7$ vertices or a complete graph on
  at least $19$ vertices.  Then there exists a constant $c>0$ such
  that for $p=p(n) \le cn^{-1/m^{(2)}(H)}$ we have ${G(n,p) \nmcarrow
    H}$ with high probability.
\end{theorem} 

In~\cite{BaCaMoPa18+}, Theorem~\ref{thm:cycles-complete-graphs} was
extended to every cycle, as it was proved that
$\pmc {C_\ell} = n^{-1/m^{(2)}(C_\ell)}$ when $\ell\geq 5$ and that
$\pmc {C_4} = n^{-3/4}\ll n^{-1/m^{(2)}(C_4)}$ (when~$\ell=3$, it is
easily seen that~$\pmc {C_3}=n^{-1}$).  In this paper we prove
that Theorem~\ref{thm:cycles-complete-graphs} can be extended to
complete graphs of order $k \ge 5$ ($\pmc{K_k} = n^{-2/(k+1)} = n^{-1/m^{(2)}(K_k)}$) 
and that $\pmc {K_4} = n^{-7/15}\ll n^{-2/5} = n^{-1/m^{(2)}(K_4)}$.

In the framework developed in~\cite{NePeSkSt14}, the probabilistic
problem of showing that with high probability the anti-Ramsey property
does not hold for $G(n,p)$ when $p\ll n^{-1/m^{(2)}(H)}$ is reduced to
a certain deterministic problem involving graphs with bounded
``density''. We will rely on this reduction to prove the following theorem.

\begin{theorem}\label{thm:main}
  For every $k\geq 5$, there exists a constant $c>0$ such that for
  $p=p(n) \le cn^{-1/m^{(2)}(K_k)}$ we have ${G(n,p)\nmcarrow K_k}$ with
  high probability.  Furthermore, $\pmc {K_4} = n^{-7/15} \ll
  n^{-1/m^{(2)}(K_4)}$.
\end{theorem}

Note that Theorems~\ref{thm:conditional} and~\ref{thm:main} together
determine $\pmc{K_k}$ for every~$k\geq4$ (note that
$\pmc{K_3}=\pmc{C_3}=n^{-1}$).
Before moving on, we remark that,
using ideas similar to those used to determine the threshold for~$K_4$
in Theorem~\ref{thm:main}, one can show that
$\pmc {K_4^-} = n^{-2/3} \ll n^{-1/2} = n^{-1/m^{(2)}(K_4^-)}$, where $K_k^-$ denotes the graph obtained
from~$K_k$ by removing an edge.

\subsection{Reduction to a deterministic problem}
 
In view of Theorem~\ref{thm:conditional}, to obtain the threshold
$\pmc {K_k}$ we must prove that for $p \ll n^{-1/m^{(2)}(K_k)}$ with
high probability there exists an edge colouring of $G \sim G(n,p)$ with no
rainbow copies of~$K_k$.

As mentioned before, the framework of Nenadov, Person, \v{S}kori\'{c},
and Steger~\cite{NePeSkSt14} reduces random Ramsey lower bound
problems to certain deterministic problems for graphs with bounded
``density'', where by ``density'' we mean the so called maximum
density.  For a graph~$H$, the \emph{maximum density}~$m(H)$ of~$H$ is
given by
\begin{equation*}
m(H)=\max\left\{{|E(J)|\over|V(J)|}\colon J\subseteq H,\;|V(J)| \geq
  1\right\}
\end{equation*}
and we remark that the threshold for the appearance of a graph $H$ in 
$G(n,p)$ is given by $n^{-1/m(H)}$  (see Theorem~\ref{thm:Bollobas}), 
which is a trivial lower bound for $\pmc H$ and for the threshold for
the property $G(n,p)\to(H)_r$.

Roughly speaking, to obtain a proper edge-colouring of $G \sim G(n,p)$ with
no rainbow $K_k$, the procedure in~\cite{NePeSkSt14} is as follows:
for every pair of disjoint edges that lie in precisely the same
$K_k$'s, give a new colour (the same colour to both edges).  Then give
new colours to edges not contained in $K_k$'s.  Let $\hat G$ be the
graph obtained by removing all edges coloured by this procedure.
In~\cite{NePeSkSt14}\footnote{In fact they prove a more general and
  stronger result.}, the authors prove that, if
$p\ll n^{-1/m^{(2)}(K_k)}$, then with high probability~$\hat G$ is
composed of a family $\mathcal{F}$ of disjoint subgraphs such that we
have $m(F) < m^{(2)}(K_k)$ for every $F\in\mathcal{F}$.  Therefore, to
find the desired colouring of~$G(n,p)$, it suffices to find such a
colouring for every graph~$F$ with $m(F) < m^{(2)}(K_k)$.  In view of
the above discussion, the part of
Theorem~\ref{thm:cycles-complete-graphs} concerning complete graphs
follows from the following result.
  
\begin{lemma}[Nenadov, Person, {\v S}kori{\'c} and
  Steger~\cite{NePeSkSt14}]
  \label{lemma:nepeskst}
  Let $H$ be a complete graph on at least $19$ vertices. Then for any
  graph $G$ with $m(G) < m^{(2)}(H)$ we have $G \nmcarrow H$.
\end{lemma}

The bound of~$19$ in Theorem~\ref{thm:cycles-complete-graphs} is
imposed only by their Lemma~\ref{lemma:nepeskst}, and hence to extend
Theorem~\ref{thm:cycles-complete-graphs} to complete graphs with at
least $5$~vertices, we give a proof of Lemma~\ref{lemma:nepeskst} that
works for all~$K_k$ with $k\geq5$.  Since the proof of
Lemma~\ref{lemma:nepeskst} in~\cite{NePeSkSt14} does not extend to
smaller complete graphs, we employ a new strategy.  Our proof will
involve separating~$G$ with $m(G) < m^{(2)}(K_k)$ into ``chains'' of
$K_k$'s and carefully constructing an edge colouring that avoids
rainbow copies of~$K_k$.

Given a graph $G$, we write $|G|$ for the number of vertices and $e(G)$ for the
number of edges of $G$.
Section~\ref{sec:large-cycles-and-complete-graphs} contains the proof
of Theorem~\ref{thm:main} for complete graphs with at least~$5$
vertices.  We remark that our proof works uniformly for all complete
graphs on at least~$5$ vertices, i.e., not only for complete graphs on
fewer than~$19$ vertices.  The proof of Theorem~\ref{thm:main}
for~$K_4$ is given in Section~\ref{sec:4vertices}.

\section{Proof of Theorem~\ref{thm:main} for $k \ge 5$}
\label{sec:large-cycles-and-complete-graphs}

In this section we prove the following lemma, which gives a lower
bound for $\pmc H$ when $H$ is a complete graph with at least five
vertices. Together with Theorem~\ref{thm:conditional}, this implies
that $\pmc {K_k} = n^{-1/m^{(2)}(H)}$.

\begin{lemma}\label{lemma:lower-large}
  Let $H$ be a complete graph on at least $5$~vertices. If $p=p(n) \ll
  n^{-1/m^{(2)}(H)}$, then ${G(n,p) \nmcarrow K_k}$ with high probability.
\end{lemma}

As discussed in the previous section, we accomplish this by proving
the following lemma.

\begin{lemma}\label{lemma:lower-density}
  Let $H$ be a complete graph on at least five vertices, then for any
  $G$ with $m(G) < m^{(2)}(H)$ we have $G \nmcarrow H$.
\end{lemma}

In the remainder of this section we prove Lemma~\ref{lemma:lower-density}.
In what follows we outline the ideas of our proof, analysing the structure of some subgraphs that will be important in our proof strategy (see Proposition~\ref{def:charac} and Definition~\ref{def:stages}).
We finish by proving an inductive result (Lemma~\ref{lemma:stages}) that directly implies Lemma~\ref{lemma:lower-density}.

In what follows, let $k\geq 5$ and let $G$ be a connected graph with
$m(G) < m^{(2)}(K_k) = (k+1)/2$.  Since we are interested in obtaining
a colouring such that every copy of $K_k$ is non-rainbow, we may
assume that all vertices and edges of $G$ are contained in a copy of
$K_k$. In this direction, we say that two copies of $K_k$ are
\emph{$K_k$-connected} if they are connected in the auxiliary graph 
that has every copy of $K_k$ as vertices and edge-intersecting copies of 
$K_k$ as edges. Furthermore, a subgraph of $G$ is a
\emph{$K_k$-component} if any edge and vertex is contained in a copy
of $K_k$ and any pair of copies of $K_k$ is \emph{$K_k$-connected}.
Clearly, we may assume that $G$ contains only a single
$K_k$-component, as $K_k$-components are edge disjoint and the
colourings of all its $K_k$-components induce a colouring of $G$.

Let $v$ be a vertex of minimum degree.
A simple but important observation is that since $m(G) < (k+1)/2$, the average degree in $G$ is less than $k+1$.
Thus, $v$ has degree at most~$k$.
The following induced subgraphs of $G$ on $v$ and some of its neighbours play a special role in our proof:
\begin{itemize}
    \item $K(v)$: induced subgraph of $G$ on $\{v\} \cup N(v)$;
    \item $R(v)$: induced subgraph of $G$ on $\{v\} \cup \{ w \in N(v)\colon$ every copy of $K_k$ containing $w$ also contains $v\}$;
    \item $S(v)$: induced subgraph of $G$ on $V(K(v)) \setminus V(R(v))$.
\end{itemize}
Furthermore, we define the following graphs:
\begin{itemize}
    \item $G_v^\ast$: the induced graph on the vertices $V(G) \setminus V(R(v))$;
    \item $G_v$: the graph obtained from $G_v^\ast$ by removing all edges not contained in a copy of $K_k$ in $G_v^\ast$.
\end{itemize}

In the inductive colouring strategy for Lemma~\ref{lemma:stages}, the
induction step will assume a colouring of $G_v$ that does not contain
a rainbow copy of $K_k$ and produce a colouring of~$G$ with the same
property.  The following simple fact provides useful information about
the structure of $G_v$.

\begin{fact}\label{fact:b_v}
Let $k\geq 5$ and let $G$ be a graph on at least $k+1$ vertices with $m(G) < (k+1)/2$ such that all vertices and edges of $G$ are contained in a copy of $K_k$.
Let $v$ be a vertex of minimum degree in $G$.
Then the following hold:
\begin{enumerate}[label=\rom]
    \item\label{fact:b_v-1} If $|G_v| \le k$ then $G_v$ is isomorphic to $K_k$;
    \item\label{fact:b_v-2} $|R(v)| \leq k-1$.
\end{enumerate}
\end{fact}
\begin{proof}
  First suppose that $|R(v)|=k+1$.  Thus, since $d(v)\leq k$ (recall
  that $v$ is a vertex of minimum degree), we know that $v$ has
  exactly $k$ neighbours.  A clique $K_k$ on $N(v)$ would contradict
  the definition of $R(v)$ so there is a non-edge in $R(v)$, say
  between vertices $u$ and $w$.  Since $w$ has degree at least $k$,
  there is an edge $\{w,z\}$ between $w$ and a vertex $z$ outside of
  $N(v)$.  However, $\{w,z\}$ is also contained in a $K_k$, so $w$
  cannot be in $R(v)$, a contradiction, so $|R(v)| \leq k$.

    For item~\ref{fact:b_v-1}, it is enough to show that any vertex that is contained in $G_v$ is contained in a $K_k$.
    Since $|G|\geq k+1$ and at most $k$ vertices are removed, at least one vertex is left.
    This vertex is contained in a copy of $K_k$ that $v$ is not contained in and, therefore, no vertex of this copy is in $R(v)$.
    If $G_v$ contains at most $k$ vertices, then it is actually isomorphic to $K_k$.

    For item~\ref{fact:b_v-2}, suppose for a contradiction that $|R(v)|=k$.
    Then, since $d(v)\leq k$, no vertex in $R(v)$ has neighbours outside of $R(v)$.
    If $d(v)=k-1$, then $G$ is a $K_k$, a contradiction with the fact that $|G|\geq k+1$.
    If $d(v)=k$, then since all vertices in $R(v)$ have degree at least $k$, the neighbourhood of $v$ induces a $K_k$, contradicting the assumption that $|R(v)|=k$.
    Therefore, $|R(v)|\leq k-1$.
\end{proof}

Note that since $d(v)\leq k$ and all vertices and edges are in a copy
of $K_k$, the subgraph $K(v)$ is isomorphic to either a $K_k$,
$K_{k+1}^-$, or $K_{k+1}$.  Indeed, if $|K(v)|=k+1$ and two edges $uw$
and $wz$ are missing, then the edge $vw$ would not be contained in a
copy of $K_k$, and if two parallel edges are missing, then $v$ would
not be contained in a copy of $K_k$ at all.  In the following
proposition we categorise $K(v)$ according to its structure.  For
brevity, in what follows, we will abuse notation and write $K(v)=K_k$
or say that $K(v)$ is a $K_k$, while we mean that $K(v)$ is isomorphic
to $K_k$.

\begin{proposition}\label{def:charac}
    Let $k\geq 5$ and let $G$ be a connected graph on at least $k+1$ vertices with $m(G) < (k+1)/2$ such that all vertices and edges of $G$ are contained in a single $K_k$-component.
    Let $v$ be a vertex of minimum degree in $G$.
    Then the subgraph $K(v)$ of $G$ has one of the configurations $X_\ell$, $Y_\ell$ or $U_1$, defined as follows:
    \begin{itemize}
        \item $X_\ell$: $K(v)=K_k$, and $R(v)=K_\ell$, and $S(v)=K_{k-\ell}$ $(1\leq \ell\leq k-2)$;
        \item $Y_\ell$: $K(v)=K^-_{k+1}$, and $R(v)=K_\ell$, and $S(v)=K^-_{k-\ell+1}$ $(1\leq \ell\leq k-2)$;
        \item $U_1$:    $K(v)=K_{k+1}$, and $R(v)=K_1$, and $S(v)=K_k$.
    \end{itemize}
\end{proposition}

\begin{proof}
Clearly, if $K(v)=K_{k+1}$, then $R(v)=K_1$ and $S(v)=K_k$, which gives us the configuration $U_1$.
So it remains to discuss the cases $K(v)=K_k$ and $K(v)=K^-_{k+1}$.

First, we will show that since $G$ is a single $K_k$-component on at least $k+1$ vertices, we have $|R(v)|\leq k-2$.
In fact, from~\ref{fact:b_v}\ref{fact:b_v-2} we already know that $|R(v)|\leq k-1$.
Suppose that $|R(v)|=k-1$.
In this case, $K(v)$ cannot be a $K_{k+1}$, so $K(v)$ is either a $K_k$ or a~$K_{k+1}^-$.
If $K(v)=K_k$, then any edge incident to vertices of $R(v)$ is in $K(v)$.
Then, $K(v)$ is not $K_k$-connected with the other copies of $K_k$ in $G$, which implies that $G$ is not a single $K_k$-component, a contradiction.
If $K(v)=K_{k+1}^-$, then $|S(v)|=2$.
Moreover, $S(v)$ is an edge as otherwise $G$ would not be a single $K_k$-component.
But then, there is a missing edge $xy$ with $x\in R(v)$ and $y\in S(v)$, as otherwise there will be a copy of $K_k$ containing the vertices of $R(v)$ and not containing $v$.
But this implies $d(x)<d(v)$, a contradiction.
Therefore, we conclude that $|R(v)|\leq k-2$.

Since $|R(v)|\leq k-2$, note that if $K(v)=K_k$, then $R(v)=K_\ell$, and $S(v)=K_{k-\ell}$ for some $1\leq \ell\leq k-2$, which is the configuration $X_\ell$.

It is left to show that if $K(v)= K^-_{k+1}$ and $R(v)$ has $\ell$ vertices (for any $1\leq \ell\leq k-2$), then $R(v)=K_\ell$, and $S(v)=K^-_{k-\ell+1}$.
Suppose that $R(v)$ is not a $K_\ell$.
Then, since there is only one missing edge in $K(v)$, we have $R(v)=K_\ell^-$, from where we conclude that there is a vertex in $R(v)$ with degree smaller than $d(v)$, a contradiction.
Then, $R(v)=K_\ell$.
Now, we just note that if $S(v)$ is not a $K^-_{k-\ell+1}$, then the missing edge $xy$ of $K(v)$ is such that $x\in R(v)$ and $y\in S(v)$, which implies $d(x)<d(v)$, a contradiction, which concludes the proof.
\end{proof}

In Figure~\ref{fig:k_5} we show all possible structures for $K(v)$ when $k=5$.
In our proof we will use the fact that $m(G)<(k+1)/2$ to bound the number of occurrences of the configurations $X_\ell, Y_\ell$, and $U_1$ as $K(v)$ in the induction.

\begin{figure}
    \centering
    \begin{subfigure}{.3\linewidth}
        \centering
        \begin{tikzpicture}[line width=1.3pt]
            \tikzstyle{every node} = [draw,circle,minimum size=2mm,inner sep=0pt,outer sep=2.5pt,fill=black];

            \node (a) at (1,0) {};
            \node (b) at (1,2) {};
            \node (c) at (2,0) {};
            \node (d) at (2,2) {};

            \node (x)[fill=white] [label=left:{$v$}] at (-1,1) {};

            \graph {%
                {[clique] (a), (b), (c), (d), (x)}
            };
        \end{tikzpicture}
        \caption{$X_1$}
    \end{subfigure}
    \begin{subfigure}{.3\linewidth}
        \centering
        \begin{tikzpicture}[line width=1.3pt]
            \tikzstyle{every node} = [draw,circle,minimum size=2mm,inner sep=0pt,outer sep=2.5pt,fill=black];

            \node (a) at (1,2) {};
            \node (b) at (1,0) {};
            \node (c) at (1.5,1) {};

            \node (x)[fill=white] [label=left:{$v$}] at (-1,1.5) {};
            \node (y)[fill=white] at (-1,.5) {};

            \graph {%
                {[clique] (a), (b), (c), (x), (y)}
            };
        \end{tikzpicture}
        \caption{$X_2$}
    \end{subfigure}
    \begin{subfigure}{.3\linewidth}
        \centering
        \begin{tikzpicture}[line width=1.3pt]
            \tikzstyle{every node} = [draw,circle,minimum size=2mm,inner sep=0pt,outer sep=2.5pt,fill=black];

            \node (a)[fill=white] at (-1,2) {};
            \node (b)[fill=white] at (-1,0) {};
            \node (c)[fill=white] [label=left:{$v$}] at (-1.5,1) {};

            \node (x) at (1,1.5) {};
            \node (y) at (1,.5) {};

            \graph {%
                {[clique] (a), (b), (c), (x), (y)}
            };
        \end{tikzpicture}
        \caption{$X_3$}
    \end{subfigure}
    \begin{subfigure}{.3\linewidth}
        \centering
        \begin{tikzpicture}[line width=1.3pt]
            \tikzstyle{every node} = [draw,circle,minimum size=2mm,inner sep=0pt,outer sep=2.5pt,fill=black];

            \node (x)[fill=white] [label=left:{$v$}] at (-1,1) {};

            \begin{scope}[shift={(1.5,1)}]
                \node (a) at (0:1) {};
                \node (b) at (72:1) {};
                \node (c) at (144:1) {};
                \node (d) at (-144:1) {};
                \node (e) at (-72:1) {};
            \end{scope}

            \graph {%
                (c) -- {(a), (b), (x),(e)},
                (b) -- {(a), (d),(e)},
                (a) -- {(d), (x),(e)},
                (d) -- {(x), (e)},
                (x) --[bend right=20](e),
                (x) --[bend left=20](b),
                (c) --[dotted](d)
            };
        \end{tikzpicture}
        \caption{$Y_1$}
    \end{subfigure}
    \begin{subfigure}{.3\linewidth}
        \centering
        \begin{tikzpicture}[line width=1.3pt]
            \tikzstyle{every node} = [draw,circle,minimum size=2mm,inner sep=0pt,outer sep=2.5pt,fill=black];

            \node (x)[fill=white] [label=left:{$v$}] at (-1,1.5) {};
            \node (y)[fill=white] at (-1,.5) {};

            \node (a) at (1,1) {};
            \node (b) at (1.5,0) {};
            \node (c) at (2,1) {};
            \node (d) at (1.5,2) {};

            \graph {%
                (c) -- {(a), (b), (d)},
                (b) -- {(a), (x), (y)},
                (a) -- {(y), (x),(d)},
                (x) -- {(y),(d)},
                (y) -- (d),
                (y) -- [bend right=10](c),
                (x) --[bend left=10](c),
                (b) --[dotted](d)
            };
        \end{tikzpicture}
        \caption{$Y_2$}
    \end{subfigure}
    \begin{subfigure}{.3\linewidth}
        \centering
        \begin{tikzpicture}[line width=1.3pt]
            \tikzstyle{every node} = [draw,circle,minimum size=2mm,inner sep=0pt,outer sep=2.5pt,fill=black];

            \node (a)[fill=white] at (-1,2) {};
            \node (b)[fill=white] at (-1,0) {};
            \node (c)[fill=white] [label=left:{$v$}] at (-1.5,1) {};

            \node (x) at (1.5,1) {};
            \node (y) at (1,0) {};
            \node (z) at (1,2) {};

            \graph {%
                {[clique] (a), (b), (c)},
                {(a), (b), (c)} --[complete bipartite] {(x), (y), (z)},
                (x) -- {(y), (z)},
                (y) --[dotted](z)
            };
        \end{tikzpicture}
        \caption{$Y_3$}
    \end{subfigure}
    \begin{subfigure}{.3\linewidth}
        \centering
        \begin{tikzpicture}[line width=1.3pt]
            \tikzstyle{every node} = [draw,circle,minimum size=2mm,inner sep=0pt,outer sep=2.5pt,fill=black];

            \node (v)[fill=white] [label=left:{$v$}] at (-1,1) {};

            \begin{scope}[shift={(1.5,1)}]
                \node (a) at (0:1) {};
                \node (b) at (72:1) {};
                \node (c) at (144:1) {};
                \node (d) at (-144:1) {};
                \node (e) at (-72:1) {};
            \end{scope}

            \graph {%
                {[clique] (a), (b), (c), (d), (e)},
                (v) -- {(a), (c), (d)},
                (v) --[bend right=20](e),
                (v) --[bend left=20](b)
            };
        \end{tikzpicture}
        \caption{$U_1$}
    \end{subfigure}
    \caption{Possible configurations of $K(v)$ for $k=5$. Dotted lines represent non-edges, the vertices of $R(v)$ are white and the vertices of $S(v)$ are black.}
\label{fig:k_5}
\end{figure}
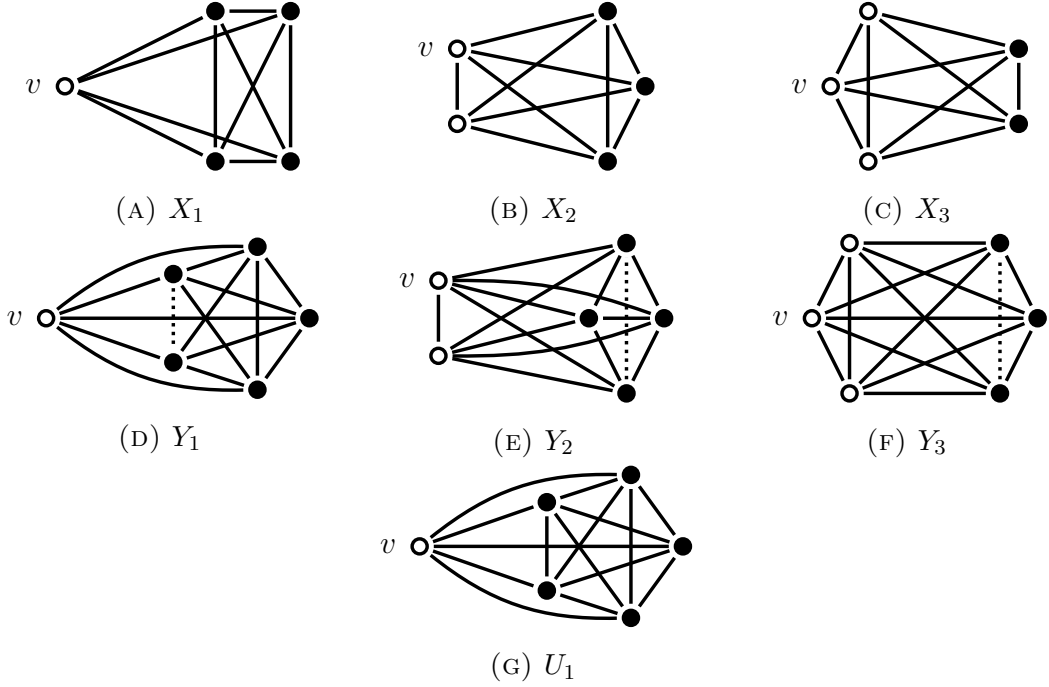

Using the characterisation given in Proposition~\ref{def:charac}, the number of vertices removed from $G$ to obtain $G_v$ is given in the subscripts of $X_\ell$, $Y_\ell$ and $U_1$, and the difference in the number of edges between $G_v^\ast$ and $G$ is as follows.
\begin{equation}\label{eq:edges}
    e(G) - e(G_v^\ast)
    =
    \begin{cases}
        k                                   & \quad \text{if } K(v) \text{ is } U_1, \\
        \binom{\ell}{2} + \ell(k-\ell)      & \quad \text{if } K(v) \text{ is } X_\ell, \\
        \binom{\ell}{2} + \ell(k-\ell+1)    & \quad \text{if } K(v) \text{ is } Y_\ell.
    \end{cases}
\end{equation}


We will use the following measure of how close $G$ is to the allowed upper bound $(k+1)/2$ on the density $m(G)$.
Set
\begin{equation*}
    b(G) := 2e(G) - (k+1)|G| + 2k.
\end{equation*}
The term $2k$ in $b(G)$ is chosen so that $b(K_k)=0$.
Moreover, from $e(G)/|G| \le m(G) < (k+1)/2$, we know that
\begin{equation}\label{eq:upperb}
    b(G) < 2k.
\end{equation}
Using~\eqref{eq:edges} we get
\begin{equation}\label{eq:badness}
    b(G) - b(G_v^\ast) 
    =
    \begin{cases}
        k-1               & \quad \text{if } K(v) \text{ is } U_1, \\
        (k-\ell-2) \ell   & \quad \text{if } K(v) \text{ is } X_\ell,\text{ for }(1\leq \ell\leq k-2), \\
        (k-\ell) \ell     & \quad \text{if } K(v) \text{ is } Y_\ell,\text{ for }(1\leq \ell\leq k-2).
    \end{cases}
\end{equation}
Note that there can be an arbitrary number of $X_{k-2}$'s in $G$ (they contribute~$0$ to $b(G)$), but because of the upper bound in~\eqref{eq:upperb} we know that all other types of $K(v)$ are limited to a small number of occurrences.
Since $(k- \ell - 2) \ell \ge k-3$ for $1 \le \ell \leq k-3$, and $(k- \ell) \ell \ge k-1$ for $1 \le \ell \leq k-2$, the following follows directly from~\eqref{eq:badness}.
\begin{equation}\label{eq:badness2}
    b(G)
    \geq
    \begin{cases}
        b(G_v^\ast) + k-1               & \quad \text{if } K(v) \text{ is } U_1 \text{ or } K(v) \text{ is }  Y_\ell,\text{ for }(1\leq \ell\leq k-2), \\
        b(G_v^\ast) + k-3   & \quad \text{if } K(v) \text{ is } X_\ell,\text{ for }(1\leq \ell\leq k-2).
    \end{cases}
\end{equation}


We will describe an inductive colouring strategy, which will always lead to an edge-colouring of $G$ with no rainbow copy of $K_k$.
To keep track of some additional properties of the colouring that will help us during the induction, we introduce five stages $P_0,\dots,P_4$, which guarantee the existence of a partial colouring of $G$ with some useful properties.

\begin{definition}[Stages]\label{def:stages}
    Let $0 \le j \le 4$.
    We say that $G$ is in stage $P_j$ or $G \in P_j$ if there exists a partial proper colouring of $G$ such that the following properties hold.
    \begin{enumerate}[label=\rom]
        \item\label{prop:non-rainbow} Any copy of $K_k$ in $G$ is non-rainbow;
        \item\label{prop:p_zero} If $G \in P_0$ then each colour is used exactly twice, in any copy of $K_k$ there are exactly two coloured edges, and any $4$ vertices span at most $3$ coloured edges;
            also, any two copies of $K_k$ intersect in at most one edge;
        \item\label{prop:four-vertices} If $G \in P_j$ $(1 \le j \le 3)$ then any $4$ vertices span at most $j+2$ coloured edges;
    \end{enumerate}
\end{definition}

Property~\ref{prop:non-rainbow} is the main property of the colouring we want to ensure.
Properties~\ref{prop:p_zero} and~\ref{prop:four-vertices} will allow us to keep the induction proof for Lemma~\ref{lemma:stages} going.
Note that, for $0\leq i\leq 3$, if $G\in P_i$ then $G\in P_{i+1}$.
We will inductively extend a partial colouring of $G_v^\ast$ to a partial colouring of $G$.
To allow such induction, we will prove that if $G_v^\ast$ is not in some stage $P_i$, then $b(G_v^\ast)$ is already ``large'', which implies that some configurations are forbidden for $K(v)$, as otherwise $b(G)$ would be too large (recall that $b(G)<2k$).

Lemma~\ref{lemma:lower-density} follows trivially from Definition~\ref{def:stages}\ref{prop:non-rainbow} and Lemma~\ref{lemma:stages} below.
\begin{lemma}\label{lemma:stages}
Let $k\geq 5$ and let $G$ be a connected graph on at least $k$ vertices with $m(G) < (k+1)/2$ such that all vertices and edges of $G$ are contained in a single $K_k$-component.
There exists $0 \le j \le 4$ and a proper partial edge-colouring of $G$ such that $G$ is in stage $P_j$ under this colouring.
Furthermore,
\begin{enumerate}[label=(\alph*)]
    \item $b(G) \geq 0$;\label{item-a}
    \item If $G\notin P_0$, then $b(G)\geq k-3$;\label{item-b}
    \item If $G\notin P_1$, then $b(G)\geq k-1$;\label{item-c}
    \item If $G\notin P_2$, then $b(G)\geq k+1$;\label{item-d}
    \item If $G\notin P_3$, then $b(G)\geq 2k-2$.\label{item-e}
\end{enumerate}
\end{lemma}
\begin{proof}
We prove the lemma by induction on the number of vertices of the graph $G$.
If $|G| = k$ then Fact~\ref{fact:b_v} implies that $G$ is a $K_k$ and then we can colour two non-adjacent edges of $G$ with the same colour, from where we conclude that $G$ is in $P_0$.
Also, $b(K_k) = 0$, so the lemma holds.

Now consider a graph $G$ on at least $k+1$ vertices satisfying the assumptions of the lemma.
Depending on $b(G)$, we have to show that $G$ is in a certain stage.
Let $v$ be a vertex of minimum degree in $G$.
Fact~\ref{fact:b_v}\ref{fact:b_v-1} implies that $G_v$ has at least $k$ vertices.
We will first handle the case that $G_v$ is a single $K_k$-component (Claim~\ref{claim:onecomponent}).
The case where there are multiple $K_k$-components will be considered in Claim~\ref{claim:multiplecomponents}.

\begin{claim}\label{claim:onecomponent}
If $G_v$ contains a single $K_k$-component, then Lemma~\ref{lemma:stages} holds.
\end{claim}

\begin{claimproof}[Proof of Claim~\ref{claim:onecomponent}]
By the inductive hypothesis, the lemma holds for $G_v$.
Let $j_v$ be the smaller index such that $G_v\in P_{j_v}$ and note that, for $1\leq j_v\leq 4$, if $G_v\in P_{j_v}$ then $G_v^*\in P_{j_v}$.
For $j_v=0$, it could be that $G_v\in P_{0}$ does not imply $G_v^*\in P_{0}$.
In fact, if $e(G_v^*\setminus G_v)>0$, then it could be that the edges in $G_v^*\setminus G_v$ form a $K_4$ in $G_v^*$ with $3$ coloured edges, because $4$ vertices of $G_v$ could span $3$ coloured edges.
Then, property $P_0$ would not hold for $G_v^*$.
But it is not hard to check that $G_v^*\in P_1$.
In view of this, we define
\begin{equation}
   j_v^*
    =
    \begin{cases}
        j_v            	& \quad \text{if } 1\leq j_v\leq 4, \\
        1		& \quad \text{if } j_v=0.
    \end{cases}
\end{equation}
For $1\leq j_v\leq 4$, as no edge in $E(G_v^\ast) \setminus E(G_v)$ is contained in  a copy of $K_k$, the partial edge-colouring that guarantees that $G_v\in P_{j_v}$ ensures that $G_v^\ast\in P_{j^\ast_v}$.
Thus, one can view $j^\ast_v$ as the smaller index such that one can always ensure that $G_v^\ast\in P_{j^\ast_v}$.

We will prove that $G$ is in some of the stages described in Definition~\ref{def:stages}.
More precisely, we will prove that
  \begin{align}\label{proof:promised}
  		&\text{if $K(v)=X_{k-2}$, then $G$ is in the same stage as $G_v^\ast$,}\nonumber\\
		&\text{if $K(v)=X_\ell$  $(1 \le \ell \leq k-3)$, then we advance at most one stage from $G_v^\ast$ to $G$,}\nonumber\\
		&\text{if $K(v)=Y_\ell$ $(1 \le \ell \leq k-2)$, then we advance at most two stages from $G_v^\ast$ to $G$},\nonumber\\
		&\text{if $K(v)=U_1$, then we advance at most two stages from $G_v^\ast$ to $G$}.
  \end{align}
  
  Note that, in the statement of the lemma, the difference in the
  bound on $b(G)$ between two consecutive stages $P_i$ and $P_{i+1}$
  is at most $k-3$ and between two stages $P_i$ and $P_{i+2}$ it is at
  most $k-1$.  From~\eqref{eq:badness2} and~\eqref{proof:promised} it
  is not hard to check that items \ref{item-a}--\ref{item-e} in the
  statement of the lemma hold.

Note that the only case that we are sure there will be no change in the stage from $G_v^\ast$ to $G$ is when $K(v)=X_{k-2}$ (recall that it contributes zero to $b(G)$).
By the difference $b(G)-b(G_v^\ast)$ described in \eqref{eq:badness} and the fact that $b(G)<2k$, the statement of the lemma applied on $G_v^\ast$ implies that
  \begin{align}\label{eq:jv}
		&\text{if }K(v)=X_\ell \ (1 \le \ell \leq k-3)\text{, then $b(G_v^\ast)\leq 2k-2$, which implies $G_v^*\in P_3$, so }j^\ast_v\leq 3,\nonumber\\
		&\text{if }K(v)=Y_\ell \ (1 \le \ell \leq k-2)\text{, then $b(G_v^\ast)\leq k+1$, which implies $G_v^*\in P_2$, so }j^\ast_v\leq 2,\\
		&\text{if }K(v)=U_1 \text{, then $b(G_v^\ast)\leq k+1$, which implies  $G_v^*\in P_2$, so } j^\ast_v\leq 2.\nonumber
  \end{align}
 
We will now give the desired partial edge-colouring that extends the edge-colouring of $G_v^\ast$ to $G$ and advances the stages in the promised way.
We split our proof into a few cases depending on the structure of $K(v)$.

\vspace{0.2cm}
\noindent\textit{Case $K(v)=X_{k-2}$.}
\vspace{0.2cm}

Since $K(v)=X_{k-2}$, the graph $G_v^\ast$ intersects $K(v)$ in exactly one edge.
We colour two disjoint edges, one contained in $R(v)$ and the other with one endpoint in $R(v)$, with a new colour.
These two edges do not close a coloured triangle
and clearly all sets of four vertices and copies of $K_k$ in $K(v)$ contain at most two coloured edges.
Also any four vertices containing one of the newly coloured edges can contain at most three coloured edges, so if $G_v^\ast\in P_0$ then Property~\ref{def:stages}\ref{prop:p_zero} holds for $G$, which implies that $G\in P_0$.
By the last part of this argument, Property~\ref{def:stages}\ref{prop:four-vertices} holds in $G$ if it did in $G_v^\ast$, so $G$ is in $P_{j^\ast_v}$.

\vspace{0.2cm}
\noindent\textit{Case $K(v)=X_\ell$ for $2 \le \ell \leq k-3$.}
\vspace{0.2cm}

By~\eqref{eq:jv}, we have $j_v^\ast\leq 3$.
We extend the current colouring in the following way:
if there is any coloured edge in $S(v)$, then we give this colour to one of the edges in $R(v)$, which contains an edge since $\ell \ge 2$.
Otherwise we choose a new colour and colour two disjoint edges that both intersect $R(v)$ with this colour.
In the first case it is trivial that $G$ is in $P_{j^\ast_v+1}$ and in the second it is easy to see  that $G$ is in $P_{j^\ast_v+1}$ as the only set of four vertices that contains the two new coloured edges has no other coloured edge.

\vspace{0.2cm}
\noindent\textit{Case $K(v)=X_1$.}
\vspace{0.2cm}

By~\eqref{eq:jv}, we have $j_v^\ast\leq 3$.
First suppose that $G_v^\ast \in P_0$.
If $S(v)$ already contains two coloured edges with the same colour, then we are done.
So assume this is not the case.
Since in $P_0$ any two copies of $K_k$ in $G_v^\ast$ intersect in at most one edge, any $K_{k-1}$ must be contained in a copy of $K_k$.
If $S(v)$ is a $K_k$, then there is a coloured edge in $S(v)$, say $e$, and its colour is used exactly twice.
Then, there is an edge incident to $v$ that we can colour with the colour of the edge $e$.
Note that any four vertices containing this newly coloured edge can contain at most three coloured edges, from where we conclude that $G$ is in $P_1$.
However, if $S(v)$ is not a $K_k$ and no edge of $S(v)$ is coloured then $e(G_v^\ast \setminus G_v)>0$ and it is enough to colour any edge of $S(v)$ and another edge incident to $v$ with a new colour and note that $G$ is in $P_2$.

Now suppose that $1\leq j_v^\ast\leq 3$.
Then, $G_v^\ast$  is in $P_1$, $P_2$ or $P_3$.
We choose an uncoloured edge $e$ in $S(v)$ and an edge $e'$ that is incident to $v$ and disjoint from $e$.
We colour $e$ and $e'$ with the same new colour.
On any four vertices not containing $v$ we increase the number of coloured edges by at most one, and any four vertices containing $v$ have at most four coloured edges.
Therefore, $G \in P_{j^\ast_v+1}$.
\vspace{0.2cm}

In the remaining cases ($K(v)=Y_\ell$ or $K(v)=U_1$), we always have $j_v^\ast\leq 2$ and we will show that we advance at most two stages.
Also note that unless $G_v^\ast$ is in stage $P_0$, we are allowed to colour two or three disjoint edges (in what follows we will use this fact repeatedly):
on any four vertices the number of coloured edges can increase by at most two, which is fine with Property~\ref{def:stages}~\ref{prop:four-vertices} as $j_v^\ast$ increases by two.
In case that $G_v^\ast$ is in $P_0$ we will separately verify that Property~\ref{def:stages}~\ref{prop:four-vertices} holds for $j=2$ in $G$, i.e., that any four vertices contain at most three edges.

\vspace{0.2cm}
\noindent\textit{Case $K(v)=Y_\ell$ for $1 \le \ell \leq k-2$.}
\vspace{0.2cm}

By~\eqref{eq:jv}, we have $j_v^\ast\leq 2$.
We have to deal with the two copies of $K_k$ contained in $K(v)$.
If $K(v)=Y_\ell$ for $2 \le \ell \leq k-2$, then there exist two disjoint edges $e$ and $e'$ incident to vertices of $R(v)$ that are contained in both copies of $K_k$.
Thus, we just give the same new colour to $e$ and $e'$, concluding that $G \in P_{j^\ast_v+2}$.

If $K(v)=Y_1$, then $S(v)=K_k^-$.
Now the two copies of $K_k$ contained in $K(v)$ intersect in a $K_{k-2}$.
If $G_v^\ast$ is in stage $P_0$, then this $K_{k-2}$ contains a triangle with vertices $\{x,y,z\}$ and hence an uncoloured edge, say $xy$.
We colour $xy$ and the edge $vz$ with the same new colour, which ensures that both copies of $K_k$ contained in $K(v)$ are non-rainbow.
Now any four vertices that contain $v$ contain at most three coloured edges;
also, for any four vertices that do not contain $v$ we only added one coloured edge so the number of coloured edges might have increased from three to four, so $G$ is in $P_2$.
Now suppose that $G_v^\ast \in P_1$.
Thus, no matter how the coloured edges are distributed, only using Property~\ref{def:stages}\ref{prop:four-vertices}, we can always find three disjoint uncoloured edges in $K(v)$ such that each of the copies of $K_k$ in $K(v)$ contains two of them.
Then, we colour these three disjoint edges with the same new colour.
Finally, suppose that $G_v^\ast \in P_2$.
It follows from Property~\ref{def:stages}\ref{prop:four-vertices} that there are two uncoloured edges $e$ and $e'$ not incident to $v$ (but not necessarily disjoint) that belong to the two copies of $K_k$ contained in $K(v)$.
For both copies of $K_k$ we can choose an additional edge incident to $v$ (say, edges $f$ and $f'$) such that colouring the edges $e$ and $f$ with the same new colour, and $e'$ and $f'$ with the same new colour (different from the colour given to $e$ and $f$) makes both copies of $K_k$ non-rainbow.
Recall that the only property we have to ensure to show that $G$ is in $P_4$ is that every copy of $K_k$ is non-rainbow.
Therefore, $G\in P_4$, which completes the case that $K(v)=Y_1$.

\vspace{0.2cm}
\noindent\textit{Case $K(v)=U_1$.}
\vspace{0.2cm}

By~\eqref{eq:jv}, we have $j_v^\ast\leq 2$.
It is easy to show that if any four vertices contain at most four coloured edges then five or more vertices contain two disjoint uncoloured edges.
Recall that for $K(v)=U_1$, the graph $S(v)$ is a $K_k$ obtained by removing $v$ from $K(v)$.
Then, Property~\ref{def:stages}\ref{prop:four-vertices} implies that there are two disjoint uncoloured edges $e$ and $e'$ in $S(v)$, which together with an edge $f$ incident to $v$ that is disjoint from $e$ and $e'$, form a set of three disjoint uncoloured edges.
We colour $e$, $e'$ and $f$ with the same new colour.
Note that if $G_v^\ast \in P_0$, all four-sets of vertices that contain two of the new coloured edges ($e$, $e'$ and $f$) either contain $v$ or is a $K_4$'s in $G_v^\ast$, so they contain at most four coloured edges in $G$, which implies that $G\in P_2$.
If $G_v^\ast$ is in $P_1$ or in $P_2$, then we observe as before that $G\in P_{j^\ast_v+2}$.
\end{claimproof}

It remains to prove that Lemma~\ref{lemma:stages} holds when $G_v$ has multiple $K_k$-components.

\begin{claim}\label{claim:multiplecomponents}
If $G_v$ contains more than one $K_k$-component, then Lemma~\ref{lemma:stages} holds.
\end{claim}

\begin{claimproof}[Proof of Claim~\ref{claim:multiplecomponents}]
  Since $G_v$ contains more than one $K_k$-component, removing $R(v)$
  from $G$ splits into edge-disjoint $K_k$-components $G_1, \dots,
  G_m$ for $m\geq 2$.  In this situation there is an extra
  complication, which is the fact that the colouring we give to the
  edges of $S(v)$ must be consistent with the colouring of the graphs
  $G_1, \dots, G_m$, which all contain some edges of $S(v)$.  On the
  other hand, large intersections of some $G_i$ with $S(v)$ contribute
  a lot to $b(G)$, which we will take advantage of.

Note that since $G_v$ contains more than one $K_k$-component, $K(v)$ is neither $X_{k-2}$ nor $U_1$, so we only have to deal with the other configurations.
We apply the induction hypothesis to each one of these $K_k$-components and, without loss of generality, we may assume that the components are vertex-disjoint in $G_v \setminus K(v)$:
intersecting in vertices would yield a denser graph and since all $K_k$-components can use different colours, combining the partial colourings would still yield a proper colouring at the vertices in which they intersect.

Recall that $b(G) = 2e(G) - (k+1)|G| + 2k$.
Let $G_{i,S(v)}$ be the subgraph of $G$ induced by the vertices that are in $S(v)\cap V(G_i)$.
Note that since $G_{i,S(v)}$ either has less than $k$ vertices or is a $K_k^-$, we have $b(G_{i,S(v)})\leq 0$.
As any component $G_i$ intersects $K(v)$ in at least one edge we get the following lower bound on $b(G)$.
\begin{align*}
b(G) 	&\ge b(K(v)) + \sum_{i=1}^{m} \big(b(G_i) - b(G_{i,S(v)})\big)\\
&= b(K(v)) + \sum_{i=1}^{m} \big(b(G_i) + |b(G_{i,S(v)})| \big).
\end{align*}
We will use this bound on $b(G)$ to limit the contributions of each $G_i$ to $b(G)$.
The following observations are helpful.
If $G_{i,S(v)}$ consists of a single edge, then by definition $|b(G_{i,S(v)})|=0$.
Moreover, one can check that
\begin{equation}\label{eq:k-3-edge}
\text{if $G_{i,S(v)}$ contains more than one edge, then $|b(G_{i,S(v)})|\geq k-3$.}
\end{equation}
If one of the $K_k$-components, say $G_1$, is in $P_0$, we will use the induction hypothesis in a slightly stronger version.
Note that for $G_1$ in $P_0$, if we repeatedly remove graphs $K(w)$ for minimum degree vertices $w$ of $G_1$, then we know that all such $K(w)$'s are $X_{k-2}$, as otherwise there would be two copies of $K_k$ sharing more than one edge, which contradicts the fact that $G_1$ is in stage $P_0$.
Then, by the colouring procedure we described before for extending $G_v^\ast$ to $G$ in case $K(v)=X_{k-2}$, we may always pick any edge $e \in G_1$ and ensure that $e$ is uncoloured and any $K_3$ containing $e$ also contains another uncoloured edge.
Thus for all components $G_i$ that are in $P_0$ and intersect $S(v)$ in a single edge $e$, we know how to give a partial colouring of $G_i$ that respects Definition~\ref{def:stages} and ensure that $e$ is uncoloured.
Alternatively, we can also guarantee that a given edge $e$ is coloured.

Furthermore, the stronger induction hypothesis also applies to the case when $G_v$ contains more than one $K_k$-component and $b(G) \le k-4$.
In general if $b(G) \le k-4$, then any two copies of $K_k$ intersect in at most a single edge and there is no cycle chain of copies of $K_k$.
This implies that there is another choice of $v$ such that $G_v$ only contains a single $K_k$-component and the above argument gives the desired statement.

Before we take care of the copies of $K_k$ that are contained in $G$ but not in $G_v$ (i.e., the copies of $K_k$ contained in $K(v)$) we deal with the combination of the colourings of the $G_i$ in $S(v)$.
For that, since there is no copy of $K_k$ in $S(v)$, we only have to check conditions \ref{prop:p_zero} and \ref{prop:four-vertices} of Definition~\ref{def:stages}.

In view of items $\ref{item-a}$--$\ref{item-e}$ of the statement of the lemma, we define $j_{\min}$ as follows, where we use $b_{\text{sum}}(G)$ for $\sum_{i=1}^{m} \big(b(G_i) + |b(G_{i,S(v)})| \big)$.
\begin{equation}\label{eq:jmin}
   j_{\min}=
    \begin{cases}
        0            	& \quad \text{if \hspace{0.75cm}} 0\leq  b_{\text{sum}}(G)< k-3, \\
        1		& \quad \text{if } k-3 \leq b_{\text{sum}}(G)< k-1, \\
        2		& \quad \text{if } k-1\leq b_{\text{sum}}(G)< k, \\
        3		& \quad \text{if } k \leq b_{\text{sum}}(G)< 2k-2.
    \end{cases}
\end{equation}

We will show that $S(v) \in P_{j_{\min}}$ (Sub-Claim~\ref{subclaim:Sv}).
Then, we deal with the copies of $K_k$ in $K(v)$ to prove that in fact we have $G \in P_{j_{\min}}$ (Sub-Claim~\ref{subclaim:G}).
In view of~\eqref{eq:jmin} it is clear that Sub-Claim~\ref{subclaim:G} implies the statement of Claim~\ref{claim:multiplecomponents}.

\begin{subclaim}\label{subclaim:Sv}
The graph $S(v)$ is in $P_{j_{\min}}$.
\end{subclaim}

\begin{claimproof}[Proof of Sub-Claim~\ref{subclaim:Sv}]
If $j_{\min}=0$, then $b_{\text{sum}}(G)< k-3$, which from~\eqref{eq:k-3-edge} implies that all $K_k$-components are in $P_0$.
Therefore, 
\begin{equation}\label{eq:propJ0}	
\text{there are no coloured edges within $S(v)$,}
\end{equation}
which trivially implies $S(v) \in P_{{0}}$.

For $j_{\min}\in\{1,2,3\}$, we only need to show that Property~\ref{prop:four-vertices} of Definition~\ref{def:stages} holds in $S(v)$, which says that any $4$ vertices span at most $j_{\min}+2$ coloured edges.
We now argue that in $S(v)$ we can not have too many coloured edges, as any coloured edge in $K(v)$ belongs to a $K_k$-component.
In fact, from the induction hypothesis,
\begin{equation}\label{eq:g1notP0lower}
\text{if $G_i\notin P_0$, then $b(G_i) \ge k-3$},
\end{equation}
and in case $G_i\in P_0$, the graph $G_{i,S(v)}$ contains more than one edge if one is coloured.
Then, from~\eqref{eq:k-3-edge}, we know that
\begin{equation}\label{eq:g1P0lower}
\text{if $G_i\in P_0$, then $|b(G_{i,S(v)})| \ge k-3$},
\end{equation}
In conclusion, every $K_k$-component $G_i$ that shares a coloured edge
with $S(v)$ contributes at least $k-3$ to $b_{\text{sum}}(G)$.

If $j_{\min}=1$ then, in view of~\eqref{eq:jmin}, 
\begin{equation}\label{eq:propJ1}	
\text{at most one of the graphs $G_i$ contributes with coloured edges to $S(v)$.}
\end{equation}
In fact, if a $K_k$-component $G_i$ contributes with coloured edges to
$S(v)$, then $G_i$ is not in $P_0$ (because of the stronger induction
hypothesis).  But then, in case there are at least two
$K_k$-components that contribute with coloured edges to $S(v)$, we
know from~\eqref{eq:k-3-edge} that $b_{\text{sum}}(G)\geq 2(k-3)\geq
k-1$, a contradiction with~\eqref{eq:jmin}.  Therefore, since
$b_{\text{sum}}(G)\leq k-1$, the induction hypothesis implies that
$G_i\in P_1$ and we are done.

If $j_{\min}=2$, then we have to argue that there can not be more than $4$ coloured edges spanned by $4$ vertices in $S(v)$.
Thus, suppose for a contradiction that $S(v)$ contains a set $S_4$ of~$4$ vertices that spans~$5$ coloured edges.
Since all $G_i$'s are in $P_2$, which implies that any~$4$ vertices span at most $4$ coloured edges (see Definition~\ref{def:stages}), if there is only one $G_i$ which contributes with coloured edges to $S(v)$, then we are done.
Thus we may assume there are at least two $G_i$'s contributing with coloured edges to $S(v)$.
But note that
\begin{equation}\label{eq:propJ2}	
\text{there are at most two $G_i$'s with coloured edges in $G_{i,S(v)}$ and they are in $P_1$,}
\end{equation}
as otherwise we would have $b_{\text{sum}}(G)\geq k$ (from~\ref{item-c},~\eqref{eq:g1notP0lower} and~\eqref{eq:g1P0lower}), which contradicts~\eqref{eq:jmin}.
Then, for any $4$ vertices in $S(v)$, each $G_{i,S(v)}$ contributes with at most $3$ coloured edges.
Suppose now that $G_1$ and $G_2$ are in~$P_0$.
Since there is no fully coloured triangle in a single $G_i$, there has to be one $G_i$ that contributes with a coloured tree on $4$ vertices in $S_4$.
Therefore,
\begin{align}\label{eq:twoP0}
b_{\text{sum}}(G) 	&\ge |b(G_{1,S(v)})| + |b(G_{2,S(v)})|\nonumber\\
		&\ge 4(k+1)  + 3(k+1) -2\cdot 6 - 4k\\
		& >k-1,\nonumber
\end{align}
a contradiction with~\eqref{eq:jmin}.
On the other hand if one of the $G_i$'s, say $G_1$, is not in $P_0$ (but $G_1$ is in $P_1$), then there might be a coloured triangle, in which case we guarantee only three vertices in each of $G_{1,S(v)}$ and $G_{2,S(v)}$, but we still get a contradiction using~\eqref{eq:g1P0lower}.
\begin{align}\label{eq:oneP0}
b_{\text{sum}}(G) 	&\ge b(G_1) + |b(G_{1,S(v)})| + |b(G_{2,S(v)})|\nonumber\\
				&\ge (k-3) + 6 (k+1)  -2 \cdot 6 - 4k\\
				&>k-1,\nonumber
\end{align}
a contradiction with~\eqref{eq:jmin}.

Finally, suppose $j_{\min}=3$, which implies from~\eqref{eq:jmin} that
$b_{\text{sum}}(G)< 2k-2$.  We aim to show that in $S(v)$ any $4$
vertices span at most $5$ coloured edges.  Similar as before suppose
for a contradiction that $S(v)$ contains a set $S_4$ of~$4$ vertices
that spans~$6$ coloured edges, i.e., it is completely coloured.  Then,
these coloured edges can not be from a single $G_{i,S(v)}$, as $G_i
\in P_3$.  It is easy to check that there are at most three $G_i$'s.
If there are exactly three of them, then they are all in~$P_1$, as
otherwise we would have from~\ref{item-c},~\eqref{eq:g1notP0lower}
and~\eqref{eq:g1P0lower} that $b_{\text{sum}}(G) \ge 2(k-3) +
(k-1)\geq 2k-2$, a contradiction with~\eqref{eq:jmin}.  If $G_1$,
$G_2$ and $G_3$ are in $P_0$, then
\begin{align*}
b_{\text{sum}}(G) 		&\ge |b(G_{1,S(v)})| + |b(G_{2,S(v)})| + |b(G_{3,S(v)})|\\
					&\ge 9 (k+1) - 6k -2 \cdot 6\\
					&\geq 2k-2,
\end{align*}
a contradiction with~\eqref{eq:jmin}.  If one of them is not in $P_0$,
similar as in case $j_{\min}=2$ we get $b_{\text{sum}}(G) \geq
(k-3)+8(k+1)-6k-12\geq 2k-2$, a contradiction.

So we may assume there are only two $K_k$-components, say $G_1$ and $G_2$.
It is easy to check that they all are in $P_2$ as otherwise we would have a contradiction with~\eqref{eq:jmin}.
If $G_1$ and $G_2$ are in $P_0$ we get $b_{\text{sum}}(G) \ge 4(k+1)  + 3(k+1) -2\cdot 6 - 4k\geq 2k-2$, similar as we did in~\eqref{eq:twoP0}.
So, we may assume w.l.o.g. that $G_1$ is not in $P_0$.
If $G_1$ is in $P_1$, then different than in case $j_{\min}=2$, even if there is a coloured triangle we guarantee that one of $G_1$ and $G_2$ contributes with a coloured tree on $4$ vertices in $S_4$.
This is because $S_4$ is fully coloured and each of them has at most $3$ coloured edges in $S_4$ (they are in $P_1$).
Then, we get 
\begin{align*}
b_{\text{sum}}(G) 	&\ge b(G_1) + |b(G_{1,S(v)})| + |b(G_{2,S(v)})|\nonumber\\
				&\ge (k-3) + 7(k+1)  -2 \cdot 6 - 4k\\
				&\geq 2k-2,\nonumber
\end{align*}
a contradiction.
So, we may assume w.l.o.g. that $G_1$ is in $P_2$.
Then,
\begin{align*}
b_{\text{sum}}(G) 	&\ge b(G_1) + |b(G_{1,S(v)})| + |b(G_{2,S(v)})|\\
				&\ge (k-1) + 6 (k+1)  -2 \cdot 6 - 4k\\
				&\geq 2k-2,\\
\end{align*}
which is again a contradiction, which concludes the proof that  $S(v)$ is in $P_{j_{\min}}$.
\end{claimproof}

It is left to prove that $G\in P_{j_{\min}}$.

\begin{subclaim}\label{subclaim:G}
The graph $G$ is in $P_{j_{\min}}$.
\end{subclaim}

\begin{claimproof}[Proof of Sub-Claim~\ref{subclaim:G}]
Since Sub-Claim~\ref{subclaim:Sv} is already proved, it remains to deal with the copies of $K_k$ contained in $K(v)$.
As in the case where $G_v$ is only a single $K_k$-component (Claim~\ref{claim:onecomponent}), we split the proof depending on the structure of $K(v)$.
Recall that since $G_v$ contains more than one $K_k$-component, $K(v)$ is neither $X_{k-2}$ nor $U_1$.

\vspace{0.2cm}
\noindent\textit{Case $K(v)=X_\ell$ for $2\leq \ell \leq k-3$.}
\vspace{0.2cm}

We proceed exactly like in the proof of Claim~\ref{claim:onecomponent}, which means that we colour an edge within $R(v)$ if there is a coloured edge in $S(v)$ or we colour two parallel edges otherwise.

\vspace{0.2cm}
\noindent\textit{Case $K(v)=Y_\ell$ for $2\leq \ell \leq k-2$.}
\vspace{0.2cm}

In this case we also proceed as in the proof of Claim~\ref{claim:onecomponent}.
We can pick two disjoint edges incident to $R(v)$ that are contained in both copies of $K_k$ in $K(v)$ and give a new colour to both of them.

\vspace{0.2cm}
\noindent\textit{Case $K(v)=X_1$.}
\vspace{0.2cm}

The case $j_{\min}=0$ was already covered earlier by the stronger
induction hypothesis, because then $b(G) \le k-4$ and $G \in P_0$.  If
$j_{\min}=1$ then either only one $G_i$ intersects $S(v)$ in more than
a single edge and all are in $P_0$ or all but one $G_i$ are in $P_0$
and each $G_i$ intersects $S(v)$ only in a single edge.  Let $G_1$ be
the special $G_i$ in both cases.  In the first case we use the
stronger induction hypothesis to ensure that there is a coloured edge
of $G_1$ in $S(v)$.  Then there is an edge incident to $v$ that is not
incident to $G_1$, which we can give the same colour.  In the latter
case we colour two disjoint edges not incident to $G_1$ and get $G \in
P_1$ as only in $G_1$ there could be a coloured triangle.

If $j_{\min}=2$ then there can be either one $K_k$-component which
contributes more than $k-3$ to $b(G)$ or two $K_k$-components that
contribute with at most $k-3$ to $b(G)$.  First, consider that there
are $K_k$-components $G_1$ and $G_2$ that contribute with at most
$k-3$ each to $b(G)$.  If there are at most $3$ coloured edges on each
set of $4$ vertices in $S(v)$, then we can proceed as in
Claim~\ref{claim:onecomponent}.  So, suppose that any set of $4$
vertices in $S(v)$ contains $4$ coloured edges.  As
in~\eqref{eq:propJ2}, there are at most two $K_k$-components $G_1$ and
$G_2$ that contribute with coloured edges to $S(v)$ and they are in
$P_1$, which implies that in any set of $4$ vertices $S$ of $S(v)$,
each of $G_1$ and $G_2$ contains only $3$ coloured edges.  Therefore,
we can use one of the colours in $S$ to colour an uncoloured edges of
$S$ keeping the colouring proper.  Now suppose, that there is only one
$K_k$-component $G_1$ contributing positively to $b(G)$.  If $G_1
\not\in P_1$, then it does not intersect $S(v)$ in more than one edge
and we can easily colour an uncoloured edge of $S(v)$ and an edge
incident to $v$ such that $G \in P_2$.  When $G_1$ intersects $S(v)$
in more than one edge we have $G_1 \in P_1$ and again easily colour
two edges such that $G \in P_2$.

If $j_{\min}=3$, then there are at most $5$ coloured edges on any set of $4$ vertices in $S(v)$.
If there are at most $4$ coloured edges on any set of $4$ vertices in $S(v)$, then we proceed as in Claim~\ref{claim:onecomponent}.
Thus, we may assume that there is a set of $4$ vertices in $S(v)$ with exactly $5$ coloured edges.
It is enough to observe that these $5$ edges cannot come from the same $K_k$-component, $G_1$ say, and that not all $G_i$ involved can contain all $4$ vertices.
Therefore, we can colour an edge incident to $v$ with the same colour as one coloured edge of $S(v)$ without any conflict.

\vspace{0.2cm}
\noindent\textit{Case $K(v)=Y_1$.}
\vspace{0.2cm}

If $K(v)=Y_1$ then we can proceed similar to the colouring given in
Claim~\ref{claim:onecomponent}.  As $b(K(v)) = b(K_{k+1}^-)=k-3$ we
have $j_{\min}\neq 0$ and for $j_{\min}\ge 1$ we consider three
vertices inside $S(v)$ that are contained in both copies of $K_k$.  We
want to colour one edge inside these three vertices and the edge
connecting the third to $v$.  For $j_{\min}= 1$ this is possible
because all $G_i$ are in $P_0$ and there are no coloured edges in
$S(v)$ so far, which gives $G \in P_1$.  When $j_{\min}=2$, there is
at most one coloured edge or a single $G_i \not\in P_0$ (which
contributes with no coloured edges to $S(v)$) and thus this is also
possible and $G \in P_2$.

Lastly, for $j_{\min}=3$, we only fail if there is a completely
coloured triangle which was created by: (i) a graph $G_1 \in P_1$ with
$b(G_1) \ge k-3$ and $|b(G_{1,S(v)})| \ge k-3$, or (ii) graphs $G_1
\in P_0$ and $G_2 \in P_0$ with $|b(G_{1,S(v)})| \ge k-3$ and
$|b(G_{2,S(v)})| \ge k-3$.  In case (i) there is no other coloured
edge but in this triangle, and therefore we can easily colour two
edges incident to $v$ with colours from this triangle to make both
copies of $K_k$ non-rainbow.  In case (ii) we can do something
similar, as $G_1$ and $G_2$ can not contain $K_{k+1}^-$ and therefore
both copies of $K_k$ contain a vertex uncovered by $G_1$ or $G_2$ that
together with $v$ can be coloured using a colour from the triangle.
\end{claimproof}
Since we proved Sub-Claims~\ref{subclaim:Sv} and~\ref{subclaim:G}, we conclude that Claim~\ref{claim:multiplecomponents} holds.
\end{claimproof}
Claims~\ref{claim:onecomponent} and~\ref{claim:multiplecomponents} imply that Lemma~\ref{lemma:stages} holds.
\end{proof}

\section{Proof of Theorem~\ref{thm:main} for $K_4$}
\label{sec:4vertices}

In this section we analyse the anti-Ramsey threshold for $K_4$, and
show that $\pmc {K_4} = n^{-7/15}$. 
For the upper bound on $\pmc {K_4}$, let $J$ be the graph obtained from $K_{3,4}$ with partition classes $\{a,b,c\}$ and $\{w,x,y,z\}$ by adding the edges $ab$, $ac$ and $bc$.
It is easy to see that in any proper colouring of $J$ there is a rainbow $K_4$.
Therefore the upper bound
\begin{equation}\label{eq:upperK4}
    \pmc {K_4} \le n^{-7/15}
\end{equation}
follows from Theorem~\ref{thm:Bollobas} below applied with $H=J$.

\begin{theorem}[Bollob\'as~\cite{Bo81}]\label{thm:Bollobas}
    Let $H$ be a fixed graph.
    Then, $p=n^{-1/m(H)}$ is the threshold for the property that $G$ contains $H$.
\end{theorem}

To show that $\pmc {K_4}\geq n^{-7/15}$ we follow a similar strategy as before, but we do not need the framework of~\cite{NePeSkSt14}, because we now have an even smaller upper bound $p \ll n^{-7/15} \ll n^{-2/(4+1)}$.

Let $G$ be a $K_4$-component with $m(G)<\frac{15}{7}$.
Observe that there always is a vertex $v$ of degree $4$ in $G$ and that the assertion of Fact~\ref{fact:b_v} still holds.
The only options for $K(v)$ are $X_1$, $X_2$ and $U_1$.
In principle, $Y_1$ and $Y_2$ would also be possible, but $Y_1$ could only occur alone and $Y_2$ is already too dense.
We define $b_{K_4}(G):=7 e(G) - 15 |G| + 18$ and note that $b_{K_4}(G)<18$ and $b_{K_4}(K_4)=0$.
Then 
\begin{equation*}
    b_{K_4}(G) - b_{K_4}(G_v) - 7 e(G_v^\ast\setminus G_v)
	=
	\begin{cases}
		6   & \quad \text{if } K(v) \text{ is } X_1, \\
		5   & \quad \text{if } K(v) \text{ is } X_2, \\
		13  & \quad \text{if } K(v) \text{ is } U_1.
	\end{cases}
\end{equation*}

Thus we can bound the number of occurrences of $X_1$, $X_2$ and $U_1$.
Configuration $X_1$ is the only case where $G_v$ could contain more than one $K_4$-component and there can be at most two different $K_4$-components, which both have one edge in common with $K(v)$.
It is thus easy to see, that any $K_4$-component $G$ with $m(G)<\tfrac{15}{7}$ contains at most $10$ vertices.

Now consider $G(n,p)$ with $p \ll n^{-7/15}$.
It follows from Markov's inequality and the union bound, that $G(n,p)$ does not contain a subgraph $G$ such that $m(G)\geq \frac{15}{7}$ and $|G|\leq 12$.
Therefore $G(n,p)$ does not contain a $K_4$-component $G$ with $m(G) \ge \frac{15}{7}$.

It remains to give the colouring of $G$ depending on the sequence of $K(v)$'s.
If $K(v)$ is $U_1$ then we are left with a single $K_4$ and it is easy to colour the whole $K_5$.
Now we claim that if $b_{K_4}(G)<6$ at most one edge is coloured in any $K_3$ and if $b_{K_4}(G)<12$ at most two edges are coloured in any $K_3$.
If $K(v)$ is $X_2$ we repeat the colour of the edge in $K(v) \cup G_v$ if that edge is coloured or otherwise we colour two new disjoint edges with a new colour, which both is fine with the above.
Only the case that $K(v)$ is $X_1$ is left to check.
If $G_v$ consists of only one $K_4$-component, then we colour one edge on the triangle $K(v) \cup G_v$ and a new edge with the same colour.
Since $X_1$ adds $6$ to $b_{K_4}(G)$ this is fine with our condition.
If $G_v$ splits in more than one $K_4$-component it is enough to observe that either we can ensure that the intersecting edges are uncoloured or we already have $b_{K_4}(G_v)>5$ and thus $b_{K_4}(G)$ will be at least $11$.

\section{Acknowledgement}
The authors would like to thank Gabriel F.~Barros for helpful comments on an earlier version of this paper and the anonymous referees for their constructive input.

\def\MR#1{}
\bibliographystyle{amsplain}
\bibliography{extracted.bib}

\end{document}